\documentclass[10pt]{amsart}
\usepackage{amssymb,amsmath,amsfonts,amscd}
\usepackage{graphicx,amssymb}
\usepackage{todonotes}
\theoremstyle{plain}
\newtheorem{theorem}{Theorem}[section]

\newtheorem{proposition}[theorem]{Proposition}
\newtheorem{definition}[theorem]{Definition}

\newtheorem{lemma}[theorem]{Lemma}

\theoremstyle{remark}
\newtheorem{remark}[theorem]{Remark}

\usepackage{pdfsync}
\usepackage{hyperref}
\hypersetup{
    colorlinks=true,
    linkcolor=blue,
    filecolor=magenta,      
    urlcolor=cyan,
    pdftitle={Overleaf Example},
    pdfpagemode=FullScreen,
    }

\author{Antonio Lerario and Francesca Tripaldi}

\begin{document}

\title[Multicomplexes on Carnot groups]
{ Multicomplexes on Carnot groups and their associated spectral sequence}
\keywords{Carnot groups, multicomplexes, spectral sequences, Rumin complex}
\subjclass{}
\maketitle
\begin{abstract}
    The aim of this paper is to give a thorough insight into the relationship between the Rumin complex on Carnot groups and the spectral sequence obtained from the filtration on forms by homogeneous weights that computes the de Rham cohomology of the underlying group. 
\end{abstract}

\section{Introduction}

Within the context of Carnot groups, an important algebraic construction that has had many different applications is that of the Rumin complex. Often denoted by $(E_0^\ast,d_c)$, this is a subcomplex of the de Rham complex $(\Omega^\ast,d)$ of the underlying Carnot group $\mathbb{G}$ which was developed as a \textit{more intrinsic} choice than the entire de Rham complex on Carnot groups.

In an intrinsic construction, one would in fact expect the stratification of the Lie algebra (see Definitions \ref{def stratification} and \ref{def carnot group}) to play a significant role. Via such a stratification and the relative homogeneous group dilations, it is possible to introduce a concept of weights of forms. In other words, the stratified structure of the underlying Carnot group reflects on the space of smooth forms $\Omega^\ast$ and it can be expressed in terms of their weights.
It is important to notice that by taking into consideration the whole de Rham complex $(\Omega^\ast,d)$, the extra structure given by the stratification is entirely missed, since all possible forms are considered. In this setting, the exterior differential $d$ can be described as a left-invariant differential operator acting on $\Omega^\ast$ whose components have various homogeneous orders for the Heisenberg calculus (see Lemma \ref{lemma split d} for a more precise formulation of this statement). This fact represents an obstacle towards the creation of  a theory of (hypo)elliptic operators on forms on Carnot groups. Indeed, the notion of hypoellipticity for left-invariant differential operators on nilpotent Lie groups relies on the Rockland condition of their principal symbols, i.e. the component of the differential operator of highest Heisenberg order \cite{fischer2016quantization}. Since in the case of the exterior differential the Rockland condition is strongly linked to the de Rham cohomology of the underlying manifold, not being able to take into account all the components of $d$ of different homogeneity is a problem \cite{dave2017graded}.

The Rumin complex $(E_0^\ast,d_c)$ was then introduced as a way to obviate this issue by constructing an appropriate subcomplex that would only include a limited amount of forms. The specific choice of which forms to use is then dictated by the the stratification of the Lie algebra. This construction was particularly successful in the case of Heisenberg groups $\mathbb{H}^{2n+1}$. In \cite{rumin1990complexe,rumin1994formes}, these intrinsic forms on $\mathbb{H}^{2n+1}$ were first defined via the following two differential ideals:
\begin{itemize}
    \item[-] $\mathcal{I}^\ast=\lbrace\gamma_1\wedge\theta+\gamma_2\wedge d\theta\mid \gamma_1,\gamma_2\in\Omega^\ast\rbrace$, the differential ideal generated by the contact form $\theta$, and
    \item[-] $\mathcal{J}^\ast=\lbrace \beta\in\Omega^\ast\mid\beta\wedge\theta=\beta\wedge d\theta=0\rbrace$ the annihilator of $\mathcal{I}^\ast$.
\end{itemize}
By simply using the properties of the Lefschetz operator $L$ given by the wedge product with the K\"ahler form $d\theta$ over $Ker\,\theta$, one can show that
\begin{align}\label{complementary quotients in heisenberg}
    \Omega^k/\mathcal{I}^k=0\;\text{ for }\;k\ge n+1\;\text{ and }\;\mathcal{J}^k=0\;\text{ for }\;k\le n\,.
\end{align}
so that the two subcomplexes $(\Omega^\ast/\mathcal{I}^\ast,d_Q)$ and $(\mathcal{J}^\ast,d_Q)$ are non-trivial in complementary degrees. Here $d_Q$ denotes the exterior differential $d$ that descends to the quotients $\Omega^\ast/\mathcal{I}^\ast$, and restricts to the subspaces $\mathcal{J}^\ast$ respectively. Again, by exploiting the properties of the Lefschetz operator $L$, it is also possible to construct a second order differential operator, which Rumin denotes as $D$, that links the two previous subcomplexes. We then obtain a new complex of intrinsic forms that computes the de Rham cohomology of the group $\mathbb{H}^{2n+1}$:
\begin{align}\label{intrinsic complex}
    \Omega^0/\mathcal{I}^0\xrightarrow[]{d_Q}\Omega^1/\mathcal{I}^1\xrightarrow[]{d_Q}\cdots\xrightarrow[]{d_Q}\Omega^n/\mathcal{I}^n\xrightarrow[]{D}\mathcal{J}^{n+1}\xrightarrow[]{d_Q}\mathcal{J}^{n+2}\xrightarrow[]{d_Q}\cdots\xrightarrow[]{d_Q}\mathcal{J}^{2n+1}\,.
\end{align}
Since these differential ideals create a filtration on smooth forms $0\subset\mathcal{J}^\ast\subset\mathcal{I}^\ast\subset\Omega^\ast$ which is stable under $d$, one can apply the machinery of spectral sequences to compute the cohomology of the de Rham complex $(\Omega^\ast,d)$. In \cite{julg1995complexe}, Julg studies precisely this construction over the Heisenberg group $\mathbb{H}^{2n+1}$ and the 7-dimensional quaternionic Heisenberg group. It should be noted that in the latter case, the filtration by differential ideals $\mathcal{J}^\ast_{l+1}\subset\mathcal{I}^\ast_{l}\subset\mathcal{J}^\ast_l$ is slightly more complicated, since there are three different contact 1-forms. In both groups, the $0^{th}$-page quotients of the form $\mathcal{J}_l^\ast/\mathcal{I}_l^\ast=E_0^{l,\ast}$ coincide with the spaces of intrinsic forms. Moreover, in the case of $\mathbb{H}^{2n+1}$, the non-trivial differentials $\partial_0$ acting on the $0^{th}$-page quotients $\Omega^k/\mathcal{I}^k$ for $k\le n$ and $\mathcal{J}^k$ for $k\ge n+1$ (see \eqref{complementary quotients in heisenberg}) coincide with the first order differential operators $d_Q$. Moreover, the only non-trivial differential $\partial_2$ on the second page quotients in degree $n$ coincides with the second order differential operator $D$, just like in the complex \eqref{intrinsic complex}. 

Even though on $\mathbb{H}^{2n+1}$ there is a clear correspondence between the Rumin differentials $d_Q$ and $D$, and the differentials $\partial_0$ and $\partial_2$ of the spectral sequence associated to the filtration $0\subset\mathcal{J}^\ast\subset\mathcal{I}^\ast\subset\Omega^\ast$, the same is not true already on the 7-dimensional quaternionic Heisenberg group. The biggest obstacle in this slightly more general case is the complication of having to deal with multiple non-trivial differentials $\partial_r$ on corresponding quotients over different pages. For example, in the case of the 7-dimensional quaternionic Heisenberg group, already on 1-forms one has that both the operators $\partial_0$ on $E_0^{0,1}=\mathcal{J}_0^1/\mathcal{I}_0^1$ and $\partial_2$ on $E_2^{0,1}$ are non-trivial.

In \cite{rumin2000around}, Rumin generalises the construction of this subcomplex of intrinsic forms from Heisenberg groups to arbitrary Carnot groups $\mathbb{G}$. The subspace of intrinsic Rumin forms $E_0^\ast$ is now defined in terms of $d_0$, the algebraic part of the exterior differential $d$, which coincides with the Chevalley–Eilenberg differential on forms. The notation $d_0$ is chosen to emphasise that this operator does not increase the weight of the form it acts on, i.e. given a differential form $\alpha$ of weight $p$, the weight of $d_0\alpha$ is still $p$ (assuming $d_0\alpha\neq 0$). Since $d_0\circ\, d_0=0$, it is possible to consider the cohomology of the complex $(\Omega^\ast,d_0)$. Once a metric on $\mathbb{G}$ is introduced, one can define the subspace $E_0^\ast=Ker\,d_0\cap\big(Im\,d_0\big)^\perp\subset\Omega^\ast$ of Rumin forms. After defining the homotopical equivalence $\Pi_E$ between $(\Omega^\ast,d)$ and a second subcomplex $(E^\ast,d)$, and the orthogonal projection $\Pi_{E_0}\colon \Omega^\ast\to E_0^\ast$, one obtains the exact subcomplex $(E_0^\ast,d_c=\Pi_{E_0}d\Pi_E)$ which is conjugated to $(E^\ast,d)$. The crucial point in this construction is the introduction of the operator $d_0^{-1}$, the inverse map of $d_0$, which can only be defined once we have a metric on $\mathbb{G}$.  

We have therefore obtained an exact subcomplex $(E_0^\ast,d_c)$ which is better adapted to the stratification of the Carnot group considered, and which computes the de Rham cohomology of the underlying manifold. Even though some connections between the Rumin complex and spectral sequences can easily be drawn, it is difficult to identify an abstract algebraic framework within which one should consider the Rumin complex in order to obtain a deeper understanding into the geometric intuition behind this construction.

By considering homogeneous weights on forms, it is possible to construct a decreasing filtration over $\Omega^\ast$ by the spaces $\mathcal{F}^p$ of forms of weight $\ge p$
\begin{align}\label{filtration by weights}
    \mathcal{F}^{T+1}=0\subset\mathcal{F}^T\subset\mathcal{F}^{T-1}\subset\cdots\subset\mathcal{F}^2\subset\mathcal{F}^1\subset\mathcal{F}^0=\Omega^\ast\,,
\end{align}
where $T$ is the homogeneous dimension of the Carnot group $\mathbb{G}$. It is easy to check that \eqref{filtration by weights} is also stable under the action of the exterior differential $d$, i.e. $d(\mathcal{F}^p)\subset\mathcal{F}^p$. It is then possible to consider its associated spectral sequence associated to compute the de Rham cohomology $(\Omega^\ast,d)$ of the underlying Carnot group. Firstly, the first page quotients $E_1^{p,\ast}$ coincide with the Rumin forms $E_0^\ast$ (once we consider them as quotients $Ker\,d_0/Im\,d_0$). Moreover, Rumin points out that in some very special cases (for example on Heisenberg groups) the Rumin differential $d_c$ coincides with a differential in the weight spectral sequence (we refer to \cite{rumin2005Palermo,tripaldi2013differential} for the precise statement).  

In this paper, we extend the study of the relationship between the Rumin differentials $d_c$ and the differentials in the weight spectral sequence in full generality. More precisely, in Theorem \ref{prop 4.7} we prove that the differential part of the operators $d_c$ coincides with the sum of the differentials $\partial_r$ that arise in the spectral sequence obtained from the filtration \eqref{filtration by weights}. We are able to do so by considering the weight filtration \eqref{filtration by weights} within the context of multicomplexes. In Section \ref{section 3}, we show how the de Rham complex $(\Omega^\ast,d)$ on an arbitrary Carnot group $\mathbb{G}$ is a multicomplex and the filtration given by its associated total complex is indeed the filtration by weights. We can then use the work of Livernet, Whitehouse and Ziegenhagen \cite{livernet2020spectral}, where they present the formulae of the differentials $\partial_r$ for $r\ge 1$ for this spectral sequence. By comparing the explicit expressions of the $\partial_r$ with the Rumin differentials $d_c$, we are able to show in Section \ref{section 4} that, morally, the Rumin differentials $d_c$ coincide with the sum $\sum_{r\ge 1}\partial_r$ of all the differentials $\partial_r$ that appear in the weight spectral sequence.   It is important to notice that in the case of the Rumin complex $(E_0^\ast,d_c)$, we are considering subspaces $E_0^\ast=Ker\,d_0\cap(Im\,d_0)^\perp$, and not quotients, and this is the central point where the introduction of a metric on $\mathbb{G}$ (or equivalently the definition of $d_0^{-1}$) becomes necessary. 

It is therefore possible to view the Rumin complex within the more abstract algebraic framework of spectral sequences. In this way, not only does it become clearer why the subcomplex $(E_0^\ast,d_c)$ computes the de Rham cohomology of the underlying manifold, but we also obtain a comprehensive bird's-eye view of the otherwise very complicated expression of the differentials $d_c$.
\section{The spectral sequence associated to a multicomplex}\label{section 2}
In this section we will present the construction of the spectral sequence associated to a multicomplex. In particular, we are interested in the explicit formulation of the differentials $\partial_r$ that appear at the $r^{th}$-step of this spectral sequence. This discussion follows the one produced in \cite{livernet2020spectral}, but with a slight modification in order to adapt the construction to the more mainstream approach adopted when discussing differential forms on Carnot groups. Indeed, in our paper we will be considering multicomplexes with differential maps $d_i\colon C\to C$ of bidegree $\vert d_i\vert=(i,1-i)$, instead of the first quadrant multicomplexes with differentials $d_i\colon C\to C$ of bidegree $\vert d_i\vert=(-i,i-1)$ that are more commonly studied in homology theory.

\begin{definition}
A multicomplex (also called a twisted chain complex) is a $(\mathbb{Z},\mathbb{Z})$-graded $k$-module $C$ equipped with maps $d_i\colon C\to C$ for $i\ge 0$ of bidegree $\vert d_i\vert=(i,1-i)$ such that
\begin{align}\label{sum=0}
    \sum_{i+j=n}d_id_j=0\;\text{ for all }\,n\ge 0\,.
\end{align}
For $C$ a multicomplex and $(a,b)\in\mathbb{Z}\times\mathbb{Z}$, we write $C_{a,b}$ for the $k$-module in bidegree $(a,b)$.
\end{definition}

If the maps $d_i\colon C\to C$ all vanish for $i\ge 1$, we obtain a chain complex with an additional grading. Instead in the case where $d_i=0$ for $i\ge 2$, we retrieve the usual notion of a bicomplex.

Given a multicomplex $C$, one could consider a priori different possible total complexes, however as already discussed in \cite{livernet2020spectral}, we will focus on the following choice for the total complex.
\begin{definition}
For a multicomplex $C$, its \textbf{associated total complex} $\mathrm{Tot}C$ is the chain complex with
\begin{align*}
    \big(\mathrm{Tot}C\big)_h=\Bigg(\prod_{\substack{a+b=h\\a\le 0}}C_{a,b}\Bigg)\oplus\Bigg(\bigoplus_{\substack{a+b=h\\a>0}}C_{a,b}\Bigg)=\Bigg(\bigoplus_{\substack{a+b=h\\b\le 0}}C_{a,b}\Bigg)\oplus\bigg(\prod_{\substack{a+b=h\\b>0}}C_{a,b}\Bigg)\,.
\end{align*}
The differential on $\mathrm{Tot}C$ is given for an arbitrary element $c\in(\mathrm{Tot}C)_h$ by:
\begin{align}\label{diffmulti}
    (dc)_a=\sum_{i\ge 0}d_i(c)_{a-i}\;.
\end{align}
Here $(c)_a$ denotes the projection of $c\in (\mathrm{Tot}C)_h$ to $C_{a,\ast}=\prod_{b,a+b=h}C_{a,b}$.
\end{definition}

In general, when working with $(\mathrm{Tot}C)_h$, it is not always possible to consider the direct product total complex $\prod_{a+b=h}C_{a,b}$ in degree $h$, as the formula \eqref{diffmulti} may involve infinite sums.
However in this paper, we will apply this construction of the total complex $\mathrm{Tot}C$ to the space of differential forms on Carnot groups. As we will see, in this case the range for $a$, which will be referred to as \textit{weights}, is finite and takes integer values between 0 and the Hausdorff dimension $Q$ of the Carnot group considered. Therefore, in this particular setting the associated total complex takes the simpler form
\begin{align}
    \big(\mathrm{Tot}C\big)_h=\bigoplus_{\substack{a+b=h\\a\ge 0}}C_{a,b}\;,\;\text{ with differential }(dc)_a=\sum_{i\ge 0}d_i(c)_{a-i}\,.
\end{align}

\begin{definition}
Given a multicomplex $C$ and its associated total complex $D\colon =\mathrm{Tot}C$, we can define for each $h$ the following subcomplexes
\begin{align}\label{filtration tot}
    \big({F}_pD\big)_h=\bigoplus_{\substack{a+b=h\\a\ge p}}C_{a,b}\,.
\end{align}
By definition, as the value of $p$ varies, the subcomplexes $F_pD$ form a filtration of $\mathrm{Tot}C$, that is $D$ is a filtered complex.
\end{definition}
\begin{remark}
It should be noted that
\begin{align*}
    F_pD=\bigoplus_{i=0}^{r-1}C_{p+i,\ast}\oplus F_{p+r}D\,,
\end{align*}
so that an arbitrary element $x\in F_pD$ can be written as
\begin{align}\label{uno}
    x=(x)_p+(x)_{p+1}+\cdots+(x)_{p+r-1}+u\,,
\end{align}
where $u\in F_{p+r}D$ and $(x)_{p+i}$ denotes the projection of $x$ to $C_{p+i,\ast}$.
\end{remark}

Let us now consider the spectral sequence associated to this filtered complex $F_pD$. For $r\ge 0$, the $r^{th}$-page of the spectral sequence is a bigraded module $E_r^{p,\ast}(D)$ with a map $\delta_r$ of bidegree $(r,1-r)$ for which $\delta_r\circ\delta_r=0$. Moreover, we have that the spaces $E_r^{p,\ast}$ can be expressed as the quotients
\begin{align*}
    E_r^{p,\ast}(D)\cong\mathcal{Z}_r^{p,\ast}(D)/\mathcal{B}_r^{p,\ast}(D)\,,
\end{align*}
where the $r$-cycles are given by
\begin{align*}
    \mathcal{Z}_r^{p,\ast}(D)\colon=F_pD\cap d^{-1}(F_{p+r}D)\,,
\end{align*}
and the $r$-boundaries are given by
\begin{align*}
    \begin{cases}
    \mathcal{B}_0^{p,\ast}(D)\colon =\mathcal{Z}_0^{p+1,\ast}(D)\;\text{ and }\\\mathcal{B}_r^{p,\ast}(D)\colon =\mathcal{Z}_{r-1}^{p+1,\ast}(D)+d\mathcal{Z}_{r-1}^{p-(r-1),\ast}(D)\;\text{ for }r\ge 1\,.
    \end{cases}
\end{align*}
Given an element $x\in \mathcal{Z}_r^{p,\ast}(D)$, we will denote by $[x]_r$ its image in $E_r^{p,\ast}(D)$, so that
\begin{align*}
    \delta_r([x]_r)=[dx]_r\;,\;\text{ for any }[x]_r\in E_r^{p,\ast}(D)\,.
\end{align*}
Expanding on the expressions for $\mathcal{Z}_r^{p,\ast}(D)$ and $\mathcal{B}_r^{p,\ast}(D)$, one can introduce the following definition.
\begin{definition}\label{def 2.6}
Let $x\in C_{p,\ast}$ and let $r\ge 1$. We define the graded submodules $Z_r^{p,\ast}$ and $B_r^{p,\ast}$ of $C_{p,\ast}$ as follows:
\begin{align}
    x\in Z_r^{p,\ast}\;\Longleftrightarrow&\;\text{ for }1\le j\le r-1\text{ there exists }z_{p+j}\in C_{p+j,\ast}\text{ such that}\notag \\&\,d_0x=0\text{ and }d_nx=\sum_{i=0}^{n-1}d_iz_{p+n-i}\text{ for all }1\le n\le r-1\label{star1}\\x\in B^{p,\ast}_r\;\Longleftrightarrow&\;\text{ for }0\le k\le r-1\,,\text{ there exists }c_{p-k}\in C_{p-k,\ast}\text{ such that}\notag\\&\begin{cases}
    x=\sum_{k=0}^{r-1}d_kc_{p-k}\;\text{ and }\\0=\sum_{k=1}^{r-1}d_{k-l}c_{p-k}\;\text{ for }1\le l\le r-1\,.
    \end{cases}\label{star2}
\end{align}
\end{definition}
\begin{remark}
In order to justify the first definition, let us consider an arbitrary element $x\in C_{p,\ast}$. Then there exists an element $\overline{x}\in F_pD$ such that $x=(\overline{x})_p$. Moreover, if we consider the expression \eqref{uno}, we will have
\begin{align}\label{overline x}
    \overline{x}=x+\xi_{p+1}+\xi_{p+2}+\cdots+\xi_{p+r-1}+\overline{u}
\end{align}
where $\xi_{p+j}=(\overline{x})_{p+j}$ for $1\le j\le r-1$ and $\overline{u}\in F_{p+r}D$. We then have that \begin{align*}
    d\overline{x}&=d\big(x+\xi_{p+1}+\xi_{p+2}+\cdots+\xi_{p+r-1}+\overline{u}\big)\\=&\underbrace{d_0x}_{C_{p,\ast}}+\underbrace{d_1x+d_0\xi_{p+1}}_{C_{p+1,\ast}}+\underbrace{d_2x+d_1\xi_{p+1}+d_0\xi_{p+2}}_{C_{p+2,\ast}}+\cdots+\\&+\underbrace{d_{r-1}x+d_{r-2}\xi_{p+1}+\cdots+d_0\xi_{p+r-1}}_{C_{p+r-1,\ast}}+\underbrace{d_rx+d_{r-1}\xi_{p+1}+\cdots+d_1\xi_{p+r-1}+d\overline{u}}_{F_{p+r}D}.
\end{align*}

Therefore $\overline{x}\in\mathcal{Z}_r^{p,\ast}$ for $r\ge 1$ if and only if $d\overline{x}\in F_{p+r}D$, that is
\begin{align*}
    &d_0x=0\\&d_1x+d_0\xi_{p+1}=0\;\rightarrow\;d_1x=d_0(-\xi_{p+1})\\&d_2x+d_1\xi_{p+1}+d_2\xi_{p+2}=0\;\rightarrow\;d_2x=d_0(-\xi_{p+2})+d_1(-\xi_{p+1})\\
    &\;\;\vdots\\
    &d_{r-1}x+d_{r-2}\xi_{p+1}+\cdots+d_0\xi_{p+r-1}=0\;\rightarrow\;d_{r-1}x=\sum_{i=0}^{r-2}d_i(-\xi_{p+r-1-i})\,.
\end{align*}
By imposing $z_{p+j}=-\xi_{p+j}\in C_{p+j,\ast}$ we then recover the expression in \eqref{star1}.

On the other hand, given an element $c\in\mathcal{Z}_{r-1}^{p-(r-1),\ast}(D)$ for $r\ge 1$, we have that
\begin{align*}
    c=(c)_{p-r+1}+(c)_{p-r+2}+\cdots+(c)_{p-1}+(c)_{p}+v
\end{align*}
with $v\in F_{p+1}D$ and $c_{p-k}\colon=(c)_{p-k}\in C_{p-k,\ast}$ for $0\le k\le r-1$, and $dc\in F_pD$. Therefore,
\begin{align*}
    dc=&\underbrace{d_0c_{p-r+1}}_{C_{p-r+1,\ast}}+\underbrace{d_1c_{p-r+1}+d_0c_{p-r+2}}_{C_{p-r+2,\ast}}+\cdots+\underbrace{d_{r-2}c_{p-r+1}+d_{r-3}c_{p-r+2}+\cdots+d_0c_{p-1}}_{C_{p-1,\ast}}+\\&+\underbrace{d_{r-1}c_{p-r+1}+d_{r-2}c_{p-r+2}+\cdots+ d_1c_{p-1}+d_0c_p}_{C_{p,\ast}}+\\&+\underbrace{d_rc_{p-r+1}+d_{r-1}c_{p-r+2}+\cdots+d_2c_{p-1}+d_1c_p+dv}_{F_{p+1}D}\,,
\end{align*}
and since an element $e\in\mathcal{Z}_{r-1}^{p+1,\ast}(D)$ is an element $e\in F_{p+1}D$, we have that $\overline{x}$ given in \eqref{overline x} will belong to $\mathcal{B}_r^{p,\ast}(D)=\mathcal{Z}_{r-1}^{p+1,\ast}(D)+d\mathcal{Z}_{r-1}^{p-(r-1),\ast}(D)$ if
\begin{align*}
    &d_0c_{p-r+1}=0\\
    &d_1c_{p-r+1}+d_0c_{p-r+2}=0\\
    &\;\;\vdots\\
    &d_{r-2}c_{p-r+1}+d_{r-3}c_{p-r+2}+\cdots+d_0c_{p-1}=0\\
    &d_{r-1}c_{p-r+1}+d_{r-2}c_{p-r+2}+\cdots+d_1c_{p-1}+d_0c_p=x
\end{align*}
which are exactly the expressions in \eqref{star2}.
\end{remark}

\begin{proposition}
For $r\ge 1$ and for all $p\in\mathbb{Z}$, we have $B_r^{p,\ast}\subseteq Z_{r}^{p,\ast}$.
\end{proposition}
\begin{proof}
Let $x\in B_r^{p,\ast}$ with $c_{p-k}\in C_{p-k,\ast}$ for $0\le k\le r-1$ satisfying equations \eqref{star2}. Define 
\begin{align*}
    z_{p+j}\colon=-\sum_{i=0}^{r-1} d_{j+i}c_{p-i}\in C_{p+j,\ast}\;\text{ for }1\le j\le r-1\,.
\end{align*}
Then 
\begin{align*}
    d_0x=d_0\bigg(\sum_{k=0}^{r-1}d_kc_{p-k}\bigg)=\sum_{n=0}^{r-1}\bigg(\sum_{i+j=n}d_id_j\bigg)c_{p-n}-\sum_{l=1}^{r-1}d_l\bigg(\sum_{k=l}^{r-1}d_{k-l}c_{p-l}\bigg)=0\,.
\end{align*}
Moreover, for an arbitrary $n\in\mathbb{N}$, we have
\begin{align*}
    d_0z_{p+n}=&-d_0\bigg(\sum_{i=0}^{r-1}d_{n+i}c_{p-i}\bigg)=-\sum_{l=0}^{r-1}\bigg(\sum_{i+j=l+n}d_id_j\bigg)c_{p-l}+\sum_{l=1}^{n-1}d_l\bigg(\sum_{k=0}^{r-1}d_{n-l+k}c_{p-k}\bigg)+\\&+\sum_{l=0}^{r-1}d_{l+n}\bigg(\sum_{i=0}^{r-1-l}d_ic_{p-l-i}\bigg)\\=&0+\sum_{l=1}^{n-1}d_l(-z_{p+n-l})+d_n\bigg(\sum_{i=0}^{r-1}d_ic_{p-i}\bigg)+\sum_{l=1}^{r-1}d_{l+n}\underbrace{\bigg(\sum_{k=l}^{r-1}d_{k-l}c_{p-k}\bigg)}_{=0\text{ by }\eqref{star2}}\\=&-\sum_{l=1}^{n-1}d_lz_{p+n-l}+d_nx
\end{align*}
so that \eqref{star1} is satisfied and hence $x\in Z_r^{p,\ast}$.
\end{proof}
\begin{proposition}\label{propo 2.8}
The map 
\begin{align*}
    \psi\colon\mathcal{Z}_r^{p,\ast}(D)/\mathcal{B}_r^{p,\ast}(D)\to Z_r^{p,\ast}/B_r^{p,\ast}
\end{align*}
sending $[x]_r$ to the class $[(x)_p]$ is well defined and it is an isomorphism.
\end{proposition}
\begin{proof}
Let us first consider the map
\begin{align*}
    \widehat{\psi}\colon\mathcal{Z}_r^{p,\ast}(D&)\to Z_r^{p,\ast}/B_r^{p,\ast}\\
    x&\longmapsto\widehat{\psi}(x)=[(x)_p]\,.
\end{align*}
Given an element $x\in\mathcal{Z}_r^{p,\ast}=F_pD\cap d^{-1}(F_{p+r}D)$, we have that in particular $dx\in F_{p+r}D$, which means that $(dx)_n=0$ for $0\le n\le r-1$. Therefore
\begin{align*}
    d_0(x)_p=0\;\text{ and }\;(dx)_{p+n}=d_n(x)_p+\sum_{i=0}^{n-1}d_i(x)_{p+n-i}=0\;\text{ for }1\le n\le r-1\,.
\end{align*}
It is then sufficient to take $z_{p+n-i}=-x_{p+n-i}$ to see that they satisfy \eqref{star1} in Definition \ref{def 2.6} and hence $(x)_p\in Z_r^{p,\ast}$. By a similar argument, one can show that $\widehat{\psi}$ is surjective.

Let us now show that $\widehat{\psi}$ is injective. Given $x=(x)_p+w\in Ker\,\widehat{\psi}$ with $w\in F_{p+1}D$, then $(x)_p\in B_r^{p,\ast}$ and hence by \eqref{star2} there exist $c_{p-k}\in C_{p-k,\ast}$ for $0\le k\le r-1$ such that
\begin{align*}
    \begin{cases}
    (x)_p=\sum_{k=0}^{r-1}d_kc_{p-k}\;\text{ and }\\
    0=\sum_{k=l}^{r-1}d_{k-l}c_{p-k}\;\text{ for }1\le l\le r-1\,.
    \end{cases}
\end{align*}
Let $c=\sum_{k=0}^{r-1}c_{p-k}\in F_{p-r+1}D$, then 
\begin{align*}
    (dc)_p=\sum_{k=0}^{r-1}d_kc_{p-k}=(x)_p\;\text{ and }\;(dc)_{p-l}=\sum_{k=l}^{r-1}d_{k-l}c_{p-k}=0\text{ for all }1\le l\le r-1\,,
\end{align*}
which implies that
\begin{align*}
    c\in F_{p-r+1}D\text{ and }dc\in F_{p}D\Longrightarrow c\in \mathcal{Z}_{r-1}^{p-r+1,\ast}D\,.
\end{align*}
Moreover, $(x)_p-dc\in F_{p+1}D$ and $x=dc+\rho$, where we take $\rho=(x)_p-dc+w\in F_{p+1}D$. Since $d^2c=0$, we have that $dx=d\rho\in F_{p+r}D$, so that $\rho\in\mathcal{Z}_{r-1}^{p+1,\ast}(D)$ and hence $Ker\,\widehat{\psi}\subseteq\mathcal{B}_r^{p,\ast}(D)$.

Conversely, if $x\in\mathcal{B}_r^{p,\ast}(D)$, then $x=\rho+dc$ for some $\rho\in\mathcal{Z}_{r-1}^{p+1,\ast}(D)$ and some $c\in\mathcal{Z}_{r-1}^{p-r+1,\ast}(D)$, so that $\rho\in F_{p+1}D$ and $dc\in F_pD$. Therefore, $(x)_p=(dc)_p$ and $(dc)_k=0$ for all $k<p$. This then implies that $(x)_p\in B_r^{p,\ast}$ and $\mathcal{B}_r^{p,\ast}(D)\subseteq Ker\,\widehat{\psi}$.
\end{proof}
\begin{theorem}
Under the isomorphism 
\begin{align*}
    \psi\colon \mathcal{Z}_r^{p,\ast}(D)/\mathcal{B}_r^{p,\ast}(D)\to Z_r^{p,\ast}/B_r^{p,\ast}
\end{align*}
studied in Proposition \ref{propo 2.8}, the $r^{th}$-differential of the spectral sequence corresponds to the following map
\begin{align*}
    \partial_r\colon Z_r^{p,\ast}/B_r^{p,\ast}\to Z_r^{p+r,\ast}/B_r^{p+r,\ast}\\
    \partial_r([x])=\bigg[d_r-\sum_{i=1}^{r-1}d_iz_{p+r-i}\bigg]
\end{align*}
where $x\in Z_r^{p,\ast}$, and the elements $z_{p+j}\in C_{p+j,\ast}$ satisfy \eqref{star1} for all $1\le j\le r-1$.
\end{theorem}
\begin{proof}
Given the elements $z_{p+j}\in C_{p+j,\ast}$ satisfy \eqref{star1}, we have that $x-z_{p+1}-z_{p+2}-\cdots-z_{p+r-1}\in F_pD$ and $d\big(x-z_{p+1}-z_{p+2}-\cdots-z_{p+r-1}\big)\in F_{p+r}D$, so that $[x-z_{p+1}-\cdots-z_{p+r-1}]_r\in\mathcal{Z}_r^{p,\ast}(D)/\mathcal{B}_r^{p,\ast}(D)$.

By Proposition \ref{propo 2.8}, we have that $\psi([x-z_{p+1}-\cdots-z_{p+r-1}]_r)=[x]$, so that
\begin{align*}
    \partial_r([x])=&\psi\circ\delta_r([x-z_{p+1}-\cdots-z_{p+r-1}]_r)=\psi([d(x-z_{p+1}-\cdots-z_{p+r-1})]_r)\\=&\big[\big(d(x-z_{p+1}-\cdots-z_{p+r-1})\big)_{p+r}\big]=\bigg[d_rx-\sum_{i=1}^{r-1}d_iz_{p+r-i}\bigg]
\end{align*}
\end{proof}

\section{Carnot groups and their multicomplexes}\label{section 3}

In this section we introduce the concept of a stratified group, or Carnot group, and present the main properties induced by its stratification, namely the concept of weights of forms and the decomposition of the exterior differential $d$ in terms of these weights. Finally, we show how these properties can be used to prove that the de Rham complex $(\Omega^\ast,d)$ is a multicomplex.
\begin{definition}\textbf{The spaces of $h$-vectors and $h$-covectors}

Given an $n$-dimensional Lie algebra $\mathfrak{g}$ with basis $\lbrace X_1,\ldots,X_n\rbrace$, we will denote its dual space by $\Lambda^1\mathfrak{g}^\ast$. This is the vector space of all linear functionals on the elements of $\mathfrak{g}$ and its also referred to as the space of 1-covectors. 

One can then consider the dual basis $\lbrace \theta_1,\ldots,\theta_n\rbrace$ for which $\langle \theta_i\mid X_j\rangle=\delta_{ij}$ for any $i,j=1,\ldots,n$. Here $\langle\cdot\mid\cdot\rangle$ denotes the duality product, that is the action of the linear functional $\theta_i\in\Lambda^1\mathfrak{g}^\ast$ on the element $X_j\in\mathfrak{g}$.  Moreover, one can also introduce an inner product $\langle\cdot,\cdot\rangle$ on $\Lambda^1\mathfrak{g}^\ast$ such that this dual basis $\lbrace\theta_1,\ldots,\theta_n\rbrace$ is orthonormal.
We will denote the exterior algebras of $\mathfrak{g}$ and $\Lambda^1\mathfrak{g}^\ast$ as
\begin{align*}
    \Lambda^\ast\mathfrak{g}=\bigoplus_{h=0}^n\Lambda^h\mathfrak{g}\;\text{ and }\;\Lambda^\ast\mathfrak{g}^\ast=\bigoplus_{h=0}^n\Lambda^h\mathfrak{g}^\ast\,,
\end{align*}
where for $1\le h\le n$ 
\begin{align*}
    &\Lambda^h\mathfrak{g}=span_\mathbb{R}\lbrace X_{i_1}\wedge\cdots\wedge X_{i_h}\mid 1\le i_1<\cdots<i_h\le n\rbrace\;\text{ and }\\&\Lambda^h\mathfrak{g}^\ast=span_\mathbb{R}\lbrace \theta_{j_1}\wedge\cdots\wedge \theta_{j_h}\mid 1\le j_1<\cdots<j_h\le n\rbrace\,.
\end{align*}
are the spaces of $h$-vectors and $h$-covectors respectively. In particular, we stress that the inner product defined on $\Lambda^1\mathfrak{g}^\ast$ extends canonically to each $\Lambda^h\mathfrak{g}^\ast$, making their bases orthonormal too.

Finally, in order to relate the spaces $\Lambda^h\mathfrak{g}$ and $\Lambda^h\mathfrak{g}^\ast$, we will also consider the following maps
\begin{align*}
    &^\ast\colon\Lambda^h\mathfrak{g}\to\Lambda^h\mathfrak{g}^\ast\;\text{ such that }\;\langle X^\ast\mid Y\rangle=\langle X,Y\rangle\;\;\forall\;Y\in\Lambda^h\mathfrak{g}\;\text{ and }\\
    &^\ast\colon\Lambda^h\mathfrak{g}^\ast\to\Lambda^h\mathfrak{g}\;\text{ such that }\;\langle \theta^\ast\mid \alpha\rangle=\langle \theta,\alpha\rangle\;\;\forall\;\alpha\in\Lambda^h\mathfrak{g}^\ast\,.
\end{align*}
In both cases, we will refer to $X^\ast\in\Lambda^h\mathfrak{g}^\ast$ and $\theta^\ast\in\Lambda^h\mathfrak{g}^\ast$ as the dual of $X\in\Lambda^h\mathfrak{g}$ and $\theta\in\Lambda^h\mathfrak{g}^\ast$ respectively.
\end{definition}

\begin{definition}\textbf{Grading}

We say that a Lie algebra $\mathfrak{g}$ is graded when it admits a vector space decomposition
\begin{align*}
    \mathfrak{g}=\bigoplus_{j=1}^\infty V_j\text{ such that }[V_i,V_j]\subset V_{i+j}
\end{align*}
and all but finitely many of the subspaces $V_j$s are $\lbrace 0\rbrace$.
\end{definition}
\begin{definition}\label{def stratification}\textbf{Stratification}

We say that a Lie algebra $\mathfrak{g}$ is stratified when $\mathfrak{g}$ admits a grading in which the first layer $V_1$ generates the whole Lie algebra.  In other words, every
element of $\mathfrak{g}$ can be written as a linear combination of iterated Lie brackets
of various elements of $V_1$. In this case, a stratification (or Carnot grading) of step $s$ for the Lie algebra $\mathfrak{g}$ can be expressed as
\begin{align}\label{stratification}
    \mathfrak{g}=V_1\oplus\cdots\oplus V_s\;,\;[V_1,V_i]=V_{i+1}\;,\;V_{s}\neq 0\;\text{ and }V_k=0\text{ for }k>s\,.
\end{align}
\end{definition}

\begin{definition}\label{def carnot group}\textbf{Carnot groups}

A stratified or Carnot group $\mathbb{G}$ of nilpotency step $s$ is a connected, simply-connected Lie group whose Lie algebra $\mathfrak{g}$ of dimension $n$ is equipped with a step $s$ stratification. 
\end{definition}

It should be stressed that in the case of a grading on $\mathfrak{g}$, one can define homogeneous dilations $\delta_\lambda\colon\mathbb{G}\to\mathbb{G}$ on the underlying Lie group for any $\lambda>0$. These dilations represent a group homomorphism and can be used to introduce a concept of homogeneous weights on the space of differential forms. Even though such homogeneous weights can be defined for any grading, since we will be exclusively working on Carnot groups $\mathbb{G}$, we will only be considering the homogeneous weights that originate from the stratification of their Lie algebra $\mathfrak{g}$.

\begin{definition}\textbf{Weights of covectors}

Given $\theta\in\Lambda^1\mathfrak{g}^\ast$, we say that $\theta$ has pure weight $p$ if $\theta^\ast\in V_p$, where $V_p$ denotes the $p^{th}$-layer of the stratification \eqref{stratification} on $\mathfrak{g}$. If this is the case, we will write $w(\theta)=p$.

In general, given an $h$-covector $\xi\in\Lambda^h\mathfrak{g}^\ast$, we say that $\xi$ as pure weight $p$ and write $w(\xi)=p$, if $\xi$ can be expressed as a linear combination of covectors $\theta_{i_1}\wedge\cdots\wedge\theta_{i_h}$ such that $w(\theta_{i_1})+\cdots+w(\theta_{i_h})=p$.
\end{definition}

\begin{remark}\label{remark basis}
In the case of a Carnot group, without loss of generality, one can consider an orthonormal basis $\lbrace X_1,\ldots,X_n\rbrace$ which is adapted to the stratification \eqref{stratification}, that is
\begin{align*}
    V_1=span_{\mathbb{R}}\lbrace X_{1},\ldots,X_{m_1}\rbrace\;\text{ and }\;V_i=span_{\mathbb{R}}\lbrace X_{m_{i-1}+1},\ldots,X_{m_i}\rbrace\text{ for }1\le i\le s\,.
\end{align*}
In particular, this implies that that the subspaces $V_i$ and $V_j$ are orthogonal whenever $i\neq j$.
Moreover, its dual basis $\lbrace\theta_1,\ldots,\theta_n\rbrace$ with $\langle\theta_i\mid X_j\rangle=\delta_{ij}$, will also be an orthonormal basis of $\Lambda^1\mathfrak{g}^\ast$ which reflects the stratification in terms of weights. In fact, for any $j=1,\ldots,m_1$ we have $w(\theta_j)=1$, and given $1\le i\le s$ we have that $w(\theta_k)=i$ for any $k=m_{i-1}+1,\ldots,m_i$.
\end{remark}

\begin{proposition}
Let us consider $\xi,\eta\in\Lambda^h\mathfrak{g}^\ast$ two arbitrary $h$-covectors. If $w(\xi)\neq w(\eta)$, then they are orthogonal, that is $\langle\xi,\eta\rangle=0$.
\end{proposition}
\begin{proof}
Let us first consider the case of $h=1$. Given $\xi,\eta\in\Lambda^1\mathfrak{g}^\ast$ such that $w(\xi)=i$ and $w(\eta)=j$ with $i\neq j$, then by definition we have $\xi^\ast\in V_i$ and $\eta^\ast\in V_j$. Therefore, $\langle\xi,\eta\rangle=\langle\xi^\ast\mid\eta\rangle=0$, and indeed, as already mentioned in the previous remark, the subspaces $V_i$ and $V_j$ are orthogonal.

If $h>1$, given $\xi,\eta\in\Lambda^h\mathfrak{g}^\ast$ with different weights, then without loss of generality one can take $\xi=\theta_{i_1}\wedge\cdots\wedge\theta_{i_h}$ and $\eta=\theta_{j_1}\wedge\cdots\wedge\theta_{j_h}$ with
\begin{align*}
    w(\xi)=w(\theta_{1_1})+\cdots+w(\theta_{i_h})\neq w(\eta)=w(\theta_{j_1})+\cdots+w(\theta_{j_h})\,.
\end{align*}
This means that there is at least an index $l\in\lbrace 1,\ldots,h\rbrace$ such that $w(\theta_{i_l})\neq w(\theta_{j_l})$, that is $\theta_{i_l}^\ast\in V_i$ and $\theta_{j_l}^\ast\in V_j$ belong to different layers and therefore 
\begin{align*}
    \langle \theta_{i_1}\wedge\cdots\wedge\theta_{i_h},\theta_{j_1}\wedge\cdots\wedge\theta_{j_h}\rangle=\langle \theta_{i_1}^\ast\wedge\cdots\wedge\theta_{i_h}^\ast,\theta_{j_1}\wedge\cdots\wedge\theta_{j_h}\rangle=0\,.
\end{align*}
\end{proof}
As a consequence, the space $\Lambda^h\mathfrak{g}^\ast$ of $h$-covectors can be expressed as a direct sum of subspaces which depend on the weight:
\begin{align*}
    \Lambda^h\mathfrak{g}^\ast=\bigoplus_{a+b=h}\Lambda^{a,b}\mathfrak{g}^\ast\,,
\end{align*}
where $\Lambda^{a,b}\mathfrak{g}^\ast$ denotes the space of $a+b=h$-covectors of weight $a$. 

Moreover, it should be noted that the nilpotency of the Carnot group $\mathbb{G}$ translates into the fact that the range of the possible weights is finite. Given its $n$-dimensional Lie algebra $\mathfrak{g}$, the maximal weight will be attained by the $n$-covectors, and $w(\theta_1\wedge\cdots\wedge\theta_n)=w(\theta_1)+\cdots+w(\theta_n)=Q$ coincides with the Hausdorff dimension of the underlying Carnot group $\mathbb{G}$. If we assign weight zero to 0-covectors $\Lambda^0\mathfrak{g}^\ast=\mathbb{R}$, then all the possible weights will be contained in the discrete set $\lbrace 0,\ldots,Q\rbrace\subset\mathbb{N}$.

All the considerations we have made so far for $h$-covectors can be extended to smooth $h$-forms, so we can express the space of smooth forms in $\mathbb{G}$ as a multicomplex. 

In the case of a Lie group $\mathbb{G}$, one can consider the subcomplex of the de Rham complex 
consisting of the left-invariant differential forms. A left-invariant $h$-form is uniquely determined
by its value at the identity, where it defines a linear map $\Lambda^h\mathfrak{g}\to\mathbb{R}$, by identifying the tangent space at the identity with the Lie algebra $\mathfrak{g}$. In other words, we can think of a left-invariant $h$-form as an element of $\Lambda^h\mathfrak{g}^\ast$.

Moreover, in the case of a connected Lie group $\mathbb{G}$, we can identify the tangent space $T_x\mathbb{G}$ to $\mathbb{G}$ at any point $x\in\mathbb{G}$ with $\mathfrak{g}$ by means of the isomorphism $dL_x$, where $L_x$ denotes the left-translation by $x$. For $\theta\in\Lambda^h\mathfrak{g}^\ast$ and $f\in C^\infty(\mathbb{G})$, we can regard $\theta\otimes f$ as a smooth $h$-form $\Omega^h$ by $(\theta\otimes f)_x=f(x)(dL_x^{-1})\theta$. This then gives rise to an isomorphism
\begin{align*}
    Hom_\mathbb{R}\big(\Lambda^h\mathfrak{g},C^\infty(\mathbb{G})\big)\cong\Lambda^h\mathfrak{g}^\ast\otimes C^\infty(\mathbb{G})\to\Gamma(\Lambda^h\mathfrak{g}^\ast)=\Omega^h\,,
\end{align*}
where $\Gamma(\Lambda^h\mathfrak{g}^\ast)$ denotes the space of smooth sections of $\Lambda^h\mathfrak{g}^\ast$.

\begin{definition}\textbf{Weights of smooth forms}

Given $\alpha\in\Omega^1$, we say that $\alpha$ has pure weight $p$ if 
\begin{align*}
    \alpha=\sum_j\theta_j\otimes f_j\;\text{ where }f_j\in C^\infty(\mathbb{G})\text{ and }\theta_j\in\Lambda^{p,-p+1}\mathfrak{g}^\ast\,,
\end{align*}
and we will write $w(\alpha)=p$.

In general, given an $h$-form $\beta\in\Omega^h$, we say that $\beta$ has pure weight $p$ and write $w(\beta)=p$, if $\beta$ can be expressed as a linear combination of $h$-forms $\theta_{j_1}\wedge\cdots\wedge\theta_{j_h}\otimes f_j$ for which each $\theta_{j_1}\wedge\cdots\wedge\theta_{j_h}$ has weight $p$ and hence belongs to $\Lambda^{p,-p+h}\mathfrak{g}^\ast$.
\end{definition}

For the sake of brevity, from now on we will use the most commonly used notation for an arbitrary smooth $h$-form $\alpha=\sum_j f_j\theta_{j_1}\wedge\cdots\wedge\theta_{j_h}$ expressed in terms of a basis of left-invariant forms $\lbrace\theta_1,\ldots,\theta_n\rbrace$ (for example the one considered in Remark \ref{remark basis}).
\begin{remark}\label{smooth forms weights}

The space of smooth forms $\Omega^\ast$ inherits the direct sum decomposition of $\Lambda^\ast\mathfrak{g}^\ast$ given by the weight. In order to highlight this relationship, we will use the following notation
\begin{align}\label{forms decomp weight}
    \Omega^h=\Gamma(\Lambda^h\mathfrak{g}^\ast)=\bigoplus_{a+b=h}\Gamma(\Lambda^{a,b}\mathfrak{g}^\ast)=\bigoplus_{a+b=h}\Omega^{a,b}\,,
\end{align}
where  $\Omega^{a,b}$ denotes the space of smooth $a+b=h$-forms of weight $a$. 

\end{remark}
\begin{lemma}\label{lemma split d}
Given a Carnot group $\mathbb{G}$ of nilpotency step $s$, the decomposition of differential forms based on weights \eqref{forms decomp weight} induces a decomposition of the exterior de Rham differential $d$ which can be easily expressed in terms of the weight increase. In fact, for an arbitrary $h$-form $\alpha\in\Omega^{p,-p+h}$ of pure weight $p$, we can write
\begin{align*}
    d\alpha=d_0\alpha+d_1\alpha+\cdots+d_s\alpha\,,
\end{align*}
where each $d_i$ denotes the part of $d$ which increases the weight of the form $\alpha$ by $i$, that is
\begin{align*}
    d_i\alpha\in\Omega^{p+i,-i-p+h+1}\;\text{ for }i=0,\ldots,s\,.
\end{align*}
\end{lemma}

\begin{proof}
Given an arbitrary $h$-form of weight $p$, $\alpha=\sum_jf_j\theta^h_j$ with $\theta_j^h\in\Lambda^{p,h-p}\mathfrak{g}^\ast$, the exterior differential applied to $\alpha$ will have the following expression: 
\begin{align*}
    d(\sum_jf_j\theta_j^h)=\sum_j\big(df_j\wedge\theta_j^h+f_j\,d\theta_j^h\big)=\sum_jdf_j\wedge\theta_j^h+\sum_jf_j\,d\theta_j^h\,.
\end{align*}

By considering the orthonormal basis $\lbrace X_1,\ldots,X_n\rbrace$ of Remark \ref{remark basis}, we obtain a very explicit expression for the first addend, that is
\begin{align*}
    \sum_jdf_j\theta_j^h=\sum_j\sum_{l=1}^nX_lf_j\,\theta_l\wedge\theta_j^h=\sum_{i=1}^s\sum_{X_l\in V_i}\sum_jX_lf_j\,\theta_l\wedge\theta_j^h=\sum_{i=1}^sd_i\alpha\,.
\end{align*}
For each $i=1,\ldots,s$ we see that
\begin{align*}
    d_i\alpha=\sum_{X_l\in V_i}\sum_{j}X_lf_j\,\theta_l\wedge\theta_j^h\in \Omega^{p+i,-i-p+h+1}\;\text{ since }X_l\in V_i\,. 
\end{align*}

Regarding the second addend, one can easily see that unless $d\theta_i^h$ vanishes, then $w(d\theta_j^h)=w(\theta_j^h)=p$, that is it keeps the weight constant. 
In the case of Carnot groups, one can prove this by using the group's dilations \cite{franchi2015differential}, however here we will follow the reasoning in \cite{tripaldi2020rumin} and use the relationship between the stratification of a Carnot group and the lower central series $\lbrace\mathfrak{g}^{(i)}\rbrace$ of its Lie algebra.

By definition of a stratification \eqref{stratification}, if we consider an orthonormal basis of $\mathfrak{g}$ adapted to the lower central series, we have the correspondence
\begin{align*}
    V_i=\mathfrak{g}^{(i-1)}\cap\big(\mathfrak{g}^{(i)}\big)^\perp\;,\;\text{ where }\mathfrak{g}^{(0)}=\mathfrak{g} \text{ and }\mathfrak{g}^{(i)}=[\mathfrak{g},\mathfrak{g}^{(i-1)}]\;,\;i\ge 1\,.
\end{align*}
This implies that a left-invariant 1-form $\theta\in\Lambda^{p,-p+1}\mathfrak{g}^\ast$ has weight $p$ if and only if $\theta=X^\ast$ with $X\in V_p=\mathfrak{g}^{(p-1)}\cap(\mathfrak{g}^{(p)})^\perp$. If we argue by contradiction, then  $d\theta=\theta_1^2+\cdots+\theta_m^2$ with at least a left-invariant 2-form $\theta_l^2\in\Lambda^2\mathfrak{g}^\ast$ with $w(\theta_l^2)\neq p$. If we then express this 2-form in terms of an orthonormal basis $\lbrace X_1,\ldots,X_n\rbrace$ adapted to the lower central series (and hence to the stratification), we get
\begin{align*}
    \theta_l^2=-c X_{l_1}^\ast\wedge X_{l_2}^\ast \text{ with }w(X_{l_1}^\ast)+w(X_{l_2}^\ast)\neq p\,.
\end{align*}
In other words, if $X_{l_1}\in V_{p_1}$ and $X_{l_2}\in V_{p_2}$, then $[X_{l_1},X_{l_2}]\in V_{p_1+p_2}\neq V_{p}$. However, by definition of the exterior differential $d$ on left-invariant forms, we get that
\begin{align*}
    -c=\langle d\theta\mid X_{l_1}\wedge X_{l_2}\rangle=-\langle \theta\mid [X_{l_1},X_{l_2}]\rangle=-\langle X,[X_{l_1},X_{l_2}]\rangle
\end{align*}
which implies that $[X_{l_1},X_{l_2}]=cX\in V_p$ and we have reached a contradiction.

The same result extends to any left-invariant $h$-form $\theta^h\in\Lambda^h\mathfrak{g}^\ast$ by using the Leibniz rule of the exterior differential $d$.

If we use the notation $d_0$ to indicate this second addend, we get
\begin{align}\label{d_0}
    d_0\alpha=\sum_jf_jd\theta_j^h\in\Omega^{p,-p+h+1}\,.
\end{align}

\end{proof}
\begin{proposition}
The de Rham complex $(\Omega^\ast,d)$ on a Carnot group of nilpotency step $s$ is a multicomplex with maps $d_i\colon\Omega^\ast\to\Omega^\ast$ of bidegree $\vert d_i\vert=(i,1-i)$ with $i=0,\ldots,s$.
\end{proposition}
\begin{proof}
As already shown in Remark \ref{smooth forms weights}, each $\Omega^{a,b}$ is a $C^\infty(\mathbb{G})$-module with bidegree $(a,b)\in\mathbb{Z}\times\mathbb{Z}$. Moreover, we have already established the existence of the differential maps $d_i\colon\Omega^\ast\to\Omega^\ast$ of bidegree $\vert d_i\vert=(i,1-i)$ in the previous lemma, so we are then left to show \eqref{sum=0} holds.

This equality follows directly from the fact that $(\Omega^\ast,d)$ is a complex, that is $d^2\alpha=0$ for any smooth form $\alpha\in\Omega^{p,-p+h}$. In fact, is we expand this formula by gathering all the different terms according to their weight we get
\begin{align*}
    d^2\alpha=&d(d_0\alpha+d_1\alpha+\cdots+d_s\alpha)\\=&(d_0+d_1+\cdots+d_s)(d_0\alpha+d_1\alpha+\cdots+d_s\alpha)\\=&d_0^2\alpha+(d_0d_1+d_1d_0)\alpha+(d_0d_2+d_1d_1+d_2d_0)\alpha+\cdots+d_s^2\alpha\\=&\sum_{n=0}^{2s}\sum_{i+j=n}d_id_j\alpha=0\,.
\end{align*}
Furthermore, for each $n=0,\ldots,2s$ we have $\sum_{i+j=n}d_id_j\alpha\in \Omega^{p+n,-p-n+h+2}$, that is each addend has different weight. Since we know that forms of different weight are orthogonal, this implies that each addend will be zero.
\end{proof}
The de Rham complex $(\Omega^\ast,d)$ on a Carnot group $\mathbb{G}$ of nilpotency step $s$ and Hausdorff dimension $Q$ is a multicomplex, and therefore we can study its associated total complex $\mathrm{Tot}\Omega$ for which
\begin{align*}
    (\mathrm{Tot}\Omega)_h=\bigoplus_{a= 0}^Q\Omega^{a,-a+h} \text{ with differential }(d\alpha)_a=\sum_{i=0}^sd_i(\alpha)_{a-i}\,.
\end{align*}
In particular, we will be interested in the spectral sequence associated to the its filtration defined in \eqref{filtration tot} which is given by
\begin{align}\label{tot filtration by weights}
    (F_p\Omega)_h=\bigoplus_{\substack{a+b=h\\a\ge p}}\Omega^{a,b}=\bigoplus_{a=p}^Q\Omega^{a,-a+h}=\lbrace \alpha\in\Omega^\ast\mid w(\alpha)\ge p\rbrace\,.
\end{align}

\section{The Rumin differentials as the differentials of the spectral sequence on this multicomplex}\label{section 4}

The purpose of this section is to shed a light into the relationship between the differential operators that appear in the Rumin complex $(E_0^\ast,d_c)$ of a Carnot group and the various differentials which appear in the various pages of the spectral sequence obtained from the weight filtration on forms.
We will not get into the details of the construction of the subcomplex $(E_0^\ast,d_c)$, and we will only give a short presentation of its main properties and focus in particular on the explicit formulation of the differentials $d_c$. For a more detailed
presentation we refer to Rumin’s paper \cite{rumin2000around} or the expository article \cite{franchi2015differential}.
\begin{definition}\textbf{The Rumin complex}

Given a Carnot group $\mathbb{G}$ of nilpotency step $s$, the Rumin complex $(E_0^\ast,d_c)$ is a subcomplex of the de Rham complex $(\Omega^\ast,d)$ where
\begin{itemize}
    \item $E_0^h=Ker\,d_0\cap\big(Im\,d_0\big)^\perp\cap\Omega^h$;
    \item $d_c=\Pi_{E_0}d\Pi_E$, where $\Pi_{E_0}=Id-d_0d_0^{-1}-d_0^{-1}d_0$ is the projection on the subspace ${E_0}$. \\A thorough explanation of $d_0^{-1}$ is given in Definition \ref{def d_0 inverse}. Moreover, the operator $d$ here is the exterior de Rham differential, and the projection $\Pi_E=Id-dPd_0^{-1}-Pd_0^{-1}d$ is defined in terms of $d_0^{-1}$, as well as the differential operator $P$ presented in Definition \ref{def P};
    \item $(E_0^\ast,d_c)$ is conjugated to the de Rham complex $(\Omega^\ast,d)$, that is it computes the same cohomology as the de Rham cohomology of the underlying Carnot group $\mathbb{G}$.
\end{itemize}

\end{definition}

The map $d_0$ in Rumin's construction is exactly the operator defined in \eqref{d_0}, i.e. the part of the exterior differential $d$ that does not change the weight of the forms. It is important to stress the fact that $d_0$ coincides with the action of $d$ on left-invariant forms, which means that not only $(\Omega^\ast,d_0)$ is a complex, but also that the quotient $Ker\,d_0/Im\,d_0$ is the Lie algebra cohomology of the Carnot group $\mathbb{G}$ with coefficients in $C^\infty(\mathbb{G})$. 

Moreover, seen within the context of the spectral sequence that originates from considering the filtration by weights \eqref{tot filtration by weights}, these quotients coincide with the quotients $E_1^{p,\ast}$ that appear from the multicomplex $(\Omega^\ast,d)$. By Proposition \ref{propo 2.8} and Definition \ref{def 2.6}, we know
\begin{align*}
    E_1^{p,\ast}= Z_1^{p,\ast}/B_1^{p,\ast}=\frac{\lbrace \alpha\in\Omega^{p,\ast}\mid d_0\alpha=0\rbrace}{\lbrace \alpha\in\Omega^{p,\ast}\mid \exists \beta\in\Omega^{p,\ast}\text{ such that }d_0\beta=\alpha\rbrace}
\end{align*}
so that
\begin{align*}
    \frac{Ker\,d_0\colon\Omega^h\to\Omega^{h+1}}{Im\,d_0\colon\Omega^{h-1}\to\Omega^h}=\bigoplus_{p=0}^QE_1^{p,-p+h}\,.
\end{align*}
Within the context of the Rumin complex however, we would like to consider these forms as subspaces of differential forms, and not quotients. In order to achieve this, it is sufficient to introduce a metric on $\mathbb{G}$ and consider instead the subspace $Ker\,d_0\cap(Im\,d_0)^\perp$. 

The notation $E_0^\ast$ to denote such subspaces was indeed inspired by the language of spectral sequences, however it may be slightly confusing within this context since each $E_0^h$ corresponds to $\bigoplus_{p=0}^QE_1^{p,-p+h}$ on the first page, and not the quotients $\bigoplus_{p=0}^QE_0^{p,-p+h}$.

\begin{definition}\label{def d_0 inverse}\textbf{The operator $d_0^{-1}$}

In order to define an ``inverse'' of the operator $d_0$, one can exploit the map
\begin{align*}
    d_0\colon\Lambda^h\mathfrak{g}^\ast/Ker\,d_0\longrightarrow\Lambda^{h+1}\mathfrak{g}^\ast\,,
\end{align*}
so that by taking an arbitrary $\beta\in\Lambda^{h+1}\mathfrak{g}^\ast$ with $\beta\neq 0$, there exists a unique $\alpha\in \Lambda^h\mathfrak{g}^\ast\cap(Ker\,d_0)^\perp$ such that $d_0\alpha=\beta+\xi$, with $\xi\in(Im\,d_0)^\perp$.

It should be noted that in general we will have $\beta\notin Im\,d_0$, so that there exist $\delta\in(Im\,d_0)^\perp$ and $\gamma\in Im\,d_0$ such that $\beta=\gamma+\delta$. For example, in the expression above $d_0\alpha=\beta+\xi$, we would have $\xi=-\delta$.

Therefore, we can define
\begin{align*}
    d_0^{-1}\colon\Lambda^{h+1}\mathfrak{g}^\ast&\longrightarrow \Lambda^h\mathfrak{g}^\ast\cap(Ker\,d_0)^\perp\\\beta&\mapsto d_0^{-1}\beta=\alpha\,.
\end{align*}
Just like we did for the operator $d_0$ in \eqref{d_0}, one can extend this operator to the space of all smooth forms $\Omega^h=\Gamma(\Lambda^h\mathfrak{g}^\ast)$.

Finally, it is important to notice that $d_0^{-1}$ also preserves the weight of the form, so that
\begin{align}\label{d_0 inverse}
    d_0^{-1}\colon\Omega^{h+1}\to\Omega^h\;,\;d_0^{-1}\big(\Omega^{a,-a+h+1}\big)\subset\Omega^{a,-a+h}\,.
\end{align}

\end{definition}
\begin{definition}\label{def P}\textbf{The operator $P$}

The definition of the operator $P$ is based on the following differential operator:
\begin{align*}
    d_0^{-1}d=d_0^{-1}\big(d_0+d_1+\cdots+d_s\big)=d_0^{-1}d_0+d_0^{-1}(d-d_0)\,,
\end{align*}
where here we are using the splitting of the exterior differential $d$ that was already presented in Lemma \ref{lemma split d}.

It is clear that by the nilpotency of the Carnot group $\mathbb{G}$ considered, there exists an $N\in\mathbb{N}$ for which $[d_0^{-1}(d-d_0)]^N\equiv 0$. We then define the differential operator $P$ as follows:
\begin{align}
    P\colon =\sum_{k=0}^N(-1)^k[d_0^{-1}(d-d_0)]^k\,.
\end{align}
\end{definition}

\begin{remark}\label{remark 2.7}
Let us point out that once we introduce a metric on $\mathbb{G}$, then for any $r\ge 1$, the condition $x\in Z_{r}^{p,\ast}$ of \eqref{star1} can be rephrased in terms of the operator $d_0^{-1}$ once we require the elements $z_{p+j}\in C_{p+j,\ast}$ to be orthogonal to the $Ker\,d_0$. 

In other words, given $x\in Z_r^{p,\ast}$, we have that there exist $z_{p+j}\in C_{p+j,\ast}\cap(Ker\,d_0)^\perp$ with $1\le j\le r-1$ such that
\begin{align*}
    z_{p+1}=&d_0^{-1}d_1x\\
    z_{p+2}=&d_0^{-1}(d_2x-d_1z_{p+1})=d_0^{-1}d_2x-d_0^{-1}d_1d_0^{-1}d_1x\\
    z_{p+3}=&d_0^{-1}(d_3x-d_2z_{p+1}-d_1z_{p+2})\\=&d_0^{-1}d_3x-d_0^{-1}d_2d_0^{-1}d_1x-d_0^{-1}d_1d_0^{-1}d_2x+d_0^{-1}d_1d_0^{-1}d_1d_0^{-1}d_1x\,.
\end{align*}
If we introduce the multi-index notation 
\begin{align*}
    (d_0^{-1}d)_{I_m^j}\colon = (d_0^{-1}d_{i_1})(d_0^{-1}d_{i_2})\cdots (d_0^{-1}d_{i_m})
\end{align*}
where $I_m^j=(i_1,\ldots,i_m)\in\mathbb{N}_+^m$ for which $\vert I_m^j\vert=i_1+i_2+\cdots+i_m=j$, then for any $j=1,\ldots,r-1$, we have the expression
\begin{align}\label{formula z via x}
    z_{p+j}=\sum_{m=1}^j(-1)^{m-1}\sum_{I_m^j}(d_0^{-1}d)_{I_m^j}x\,.
\end{align}
For example, in the case where $j=4$, we would have
\begin{align*}
    z_{p+4}=&\sum_{m=1}^4(-1)^{m-1}\sum_{I_m^4}(d_0^{-1}d)_{I_m^4}x\\=&d_0^{-1}d_4x-(d_0^{-1}d_1d_0^{-1}d_3+d_0^{-1}d_2d_0^{-1}d_2+d_0^{-1}d_3d_0^{-1}d_1)x+\\&+(d_0^{-1}d_2d_0^{-1}d_1d_0^{-1}d_1+d_0^{-1}d_1d_0^{-1}d_2d_0^{-1}d_1+d_0^{-1}d_1d_0^{-1}d_1d_0^{-1}d_2)x+\\&-(d_0^{-1}d_1d_0^{-1}d_1d_0^{-1}d_1d_0^{-1}d_1)x\,.
\end{align*}
\end{remark}
\begin{remark}
By using the explicit expression \eqref{formula z via x} of the $z_{p+j}\in C_{p+j,\ast}\cap\big(Ker\,d_0\big)^\perp$ in terms of $x$ found in Remark \ref{remark 2.7}, it is possible to obtain the following expression for the $r^{th}$-differential of the spectral sequence:
\begin{align}\label{rth diff d_0 inverse}
    \partial_r([x])=\bigg[d_rx-\sum_{i=1}^{r-1}d_i\bigg(\sum_{m=1}^{r-i}(-1)^{m-1}\sum_{I_m^{r-i}}(d_0^{-1}d)_{I_m^{r-i}}\bigg)x\bigg]\,.
\end{align}
For example, for $r=2$ we will have
\begin{align*}
    \partial_2[x]=\bigg[d_2x-d_1\bigg((-1)^0\sum_{I_1^1}(d_0^{-1}d)_{I_1^1}\bigg)x\bigg]=\big[d_2x-d_1d_0^{-1}d_1x\big]\,.
\end{align*}
In the case of $r=3$, we will have
\begin{align*}
    \partial_3[x]=&\bigg[d_3x-d_1\bigg(\sum_{m=1}^2(-1)^{m-1}\sum_{I_m^2}(d_0^{-1}d)_{I_m^2}\bigg)x-d_2\bigg((-1)^0\sum_{I_1^1}(d_0^{-1}d)_{I_1^1}\bigg)x\bigg]\\=&\bigg[d_3x-d_1\sum_{I_1^2}(d_0^{-1}d)_{I_1^2}x+d_1\sum_{I_2^2}(d_0^{-1}d)_{I_2^2}x-d_2d_0^{-1}d_1x\bigg]\\=&\big[d_3x-d_1d_0^{-1}d_2x+d_1d_0^{-1}d_1d_0^{-1}d_1x-d_2d_0^{-1}d_1x\big]\,.
\end{align*}
\end{remark}

\begin{theorem}\label{prop 4.7}
Given an arbitrary Carnot group $\mathbb{G}$, the differential part of the Rumin differentials $d_c$ coincides with the sum of the differentials $\partial_r$ that appear in the multicomplex spectral sequence generated by considering the filtration by weights over the space of smooth forms.
\end{theorem}
\begin{proof}
The claim follows once we consider the explicit expression of the Rumin differentials 
\begin{align*}
    d_c\colon E_0^h\to E_0^{h+1}\,.
\end{align*}

Given an arbitrary $h$-form $\alpha\in\Omega^h$, we will have that
\begin{align*}
    d_c\big(\Pi_{E_0}\alpha\big)=\Pi_{E_0}d\Pi_E\Pi_{E_0}\alpha\,.
\end{align*}
If $\alpha'=\Pi_{E_0}\alpha\in E_0^h=Ker\,d_0\cap(Im\,d_0)^{\perp}\cap\Omega^h$, we have $d_0\alpha'=d_0^{-1}\alpha'=0$ so that
\begin{align*}
    \Pi_E\alpha'=&(Id-dPd_0^{-1}-Pd_0^{-1}d)\alpha'=\alpha'-dP{d_0^{-1}\alpha'}-Pd_0^{-1}d \alpha'\\=&\alpha'-\sum_{k=0}^N(-1)^k[d_0^{-1}(d-d_0)]^kd_0^{-1}d\alpha'\\=&\alpha'-d_0^{-1}d\alpha'+d_0^{-1}(d-d_0)d_0^{-1}d\alpha'-[d_0^{-1}(d-d_0)]^2d_0^{-1}d\alpha'+\cdots\\=&\alpha'-d_0^{-1}(d-d_0)\alpha'+[d_0^{-1}(d-d_0)]^2\alpha'-[d_0^{-1}(d-d_0)]^3\alpha'+\cdots
\end{align*}

Moreover, given an arbitrary form $\beta\in\Omega^{h}\cap Im\,d_0^{-1}$, that is $\beta=d_0^{-1}\xi$ for some $\xi\in \Omega^{h+1}$, we have that
\begin{align*}
    \Pi_{E_0}d\big(d_0^{-1}\xi\big)=&\Pi_{E_0}d\beta=\Pi_{E_0}d_0\beta+\Pi_{E_0}(d-d_0)\beta\\=&(Id-d_0d_0^{-1}-d_0^{-1}d_0)d_0d_0^{-1}\xi+\Pi_{E_0}(d-d_0)\beta\\=&d_0d_0^{-1}\xi-d_0d_0^{-1}d_0d_0^{-1}\xi+\Pi_{E_0}(d-d_0)\beta=\Pi_{E_0}(d-d_0)\beta\\=&\Pi_{E_0}(d-d_0)\big(d_0^{-1}\xi\big)\,,
\end{align*}
so that
\begin{align*}
    \Pi_{E_0}&d\Pi_E\alpha'=\Pi_{E_0}d\big(\alpha'-\sum_{k=0}^N(-1)^k[d_0^{-1}(d-d_0)]^kd_0^{-1}d\alpha'\big)\\=&\Pi_{E_0}(d-d_0)\big[\alpha'-d_0^{-1}(d-d_0)\alpha'+[d_0^{-1}(d-d_0)]^2\alpha'-[d_0^{-1}(d-d_0)]^3\alpha'+\cdots\big]\\=&\Pi_{E_0}\big(d_1+d_2-d_1d_0^{-1}d_1+d_3-d_1d_0^{-1}d_2-d_2d_0^{-1}d_1+d_1d_0^{-1}d_1d_0^{-1}d_1+\cdots\big)\alpha'\\=&\Pi_{E_0}\,\sum_{r=1}^N\bigg[d_r-\sum_{i=1}^{r-1}d_i\bigg(\sum_{m=1}^{r-i}(-1)^{m-1}\sum_{I_m^{r-i}}(d_0^{-1}d)_{I_m^{r-i}}\bigg)\bigg]\Pi_{E_0}\alpha\,,
\end{align*}
where we are using the same multi-index notation that was introduced in Remark \ref{remark 2.7}.

If we compare this final formula with the explicit expression for the $r^{th}$-differentials of the spectral sequence \eqref{rth diff d_0 inverse}, we can see the clear relationship between the operator $d\Pi_E$ applied to $\Pi_{E_0}\alpha$ and the differentials $\partial_r$.

It is however crucial to highlight that even though the explicit formulation of the operators coincide, the operators $d_c$ and $\partial_r$ act on different spaces.

As already pointed out, we have a clear correspondence between Rumin forms $E_0^h=Ker\,d_0\cap(Im\,d_0)^\perp$ and $\bigoplus_p E_1^{p,\ast}$. Indeed, in the special case of a form $\alpha\in Z_1^{p,-p+h}$ (an $h$-form of weight $p$ that belongs to $Ker\,d_0$) for which the Rumin differential of $\Pi_{E_0}\alpha$ has only differential order 1, then
\begin{itemize}
    \item $\alpha\in Z_1^{p,-p+h}$;
    \item $\Pi_{E_0}\alpha\in Z_1^{p,-p+h}\cap\big(B_1^{p,-p+h}\big)^\perp=E_0^{h}\cap\Omega^{p,-p+h}$;
    \item $d_c\Pi_{E_0}\alpha=\Pi_{E_0}d_1\Pi_{E_0}\alpha\in Z_1^{p+1,-p+h}\cap(B_1^{p+1,-p+h})^\perp=E_0^{h+1}\cap\Omega^{p+1,-p+h}$.
\end{itemize}
Compare this to the action of $\partial_1$ on the same $\alpha\in Z_1^{p,-p+h}$:
\begin{itemize}
    \item $\alpha\in Z_1^{p,-p+h}$;
    \item $[\alpha]\in Z_1^{p,-p+h}/B_1^{p,-p+h}=E_1^{p,-p+h}$;
    \item $\partial_1([\alpha])=\big[d_1\alpha\big]\in E_1^{p+1,-p+h}=Z_1^{p+1,-p+h}/B_1^{p+1,-p+h}$.
\end{itemize}
In order to consider the more general case, let us first point out that by Definition \ref{def 2.6}, we have the following set of inclusions
\begin{align*}
    B_1^{p,\ast}\subset B_2^{p,\ast}\subset\cdots\subset B_{l}^{p,\ast}\subset B_{l+1}^{p,\ast}\subset\cdots\\Z_1^{p,\ast}\supset Z_2^{p,\ast}\supset\cdots\supset Z_{l}^{p,\ast}\supset Z_{l+1}^{p,\ast}\supset\cdots
\end{align*}
Let us now assume that for a given form $\alpha\in Z_1^{p,-p+h}$ there exist $l_1,l_2,\ldots,l_k\in\mathbb{N}$ with $l_1<l_2<\cdots<l_k$ for which
\begin{align*}
    d_c\Pi_{E_0}\alpha\in \bigoplus_{i=1}^k\Omega^{p+l_i,-p-l_i+h+1}\,.
\end{align*}
This then implies in particular that $\alpha\in Z_{l_k}$, and hence it makes sense to compute the action of $\partial_i[\alpha]$ for $i=1,\ldots,l_k$. However for each $i$ we have that $[\alpha]$ is taken in a different quotient space $E_i^{p,-p+h}$:
\begin{itemize}
    \item $\alpha\in Z_{i}^{p,-p+h}\subset Z_1^{p,-p+h}$;
    \item $[\alpha]\in E_{i}^{p,-p+h}=Z_{i}^{p,-p+h}/B_{i}^{p,-p+h}\subset E_1^{p,-p+h}$;
    \item $\partial_i([\alpha])\in E_i^{p+i,-p-i+h+1}\subset E_1^{p+i,-p-i+h+1}$.
\end{itemize}
whereas the operator $d\Pi_{E}$ acts on $\Pi_{E_0}\alpha\in Z_1^{p,-p+h}\cap(B_1^{p,-p+h})^\perp$.

\end{proof}


\begin{remark}

Morally, the idea behind the spectral sequence construction is to compute the de Rham cohomology of the underlying Carnot group by taking progressively smaller subcomplexes, whereas in the Rumin complex we consider all of these subcomplexes at the same time as acting on the first page of the spectral sequence.
\end{remark}

\section*{Acknowledgments}

F. Tripaldi is supported by Swiss National Science Foundation Grant nr 200020-191978.

		\bigskip
		\bibliographystyle{amsplain}

\bibliography{bibli}

\bigskip

\tiny{
\noindent
Antonio Lerario 
\par\noindent SISSA
\par\noindent Mathematics Area
\par\noindent via Bonomea, 265
34146, Trieste, Italy
\par\noindent
e-mail:lerario@sissa.it
}

\medskip

\tiny{
\noindent
Francesca Tripaldi 
\par\noindent Mathematisches Institut,
\par\noindent University of Bern, 
\par\noindent Sidlerstrasse 5
3012 Bern, Switzerland.
\par\noindent
e-mail: francesca.tripaldi@unibe.ch
}

\end{document}